\documentclass{amsart}
\usepackage{amsthm,amssymb,amsmath}
\usepackage{epsfig}
\usepackage{color}

\newtheorem{prelem}{{\bf Proposition}}

\newtheorem{theorem}{Theorem}
\newtheorem{corollary}[theorem]{Corollary}

\newtheorem{observation}[theorem]{Observation}

\theoremstyle{definition}

\theoremstyle{remark}

\input{tcilatex}

\begin{document}
\title{The $k$-Tuple domatic Number of a graph}
\author{ Adel P. Kazemi \vspace{4mm}\\
Department of Mathematics\\
University of Mohaghegh Ardabili\\
P. O. Box 5619911367, Ardabil, Iran\\
adelpkazemi@yahoo.com\vspace{3mm} \\
}
\date{}
\maketitle

\begin{abstract}
For every positive integer $k$, a set $S$ of vertices in a graph
$G=(V,E)$ is a $k$-tuple dominating set of $G$ if every vertex of
$V-S$ is adjacent to least $k$ vertices and every vertex of $S$ is
adjacent to least $k-1$ vertices in $S$. The minimum cardinality of
a $k$-tuple dominating set of $G$ is the $k$-tuple domination number
of $G$. When $k=1$, a $k$-tuple domination number is the
well-studied domination number. We define the $k$-tuple domatic
number of $G$ as the largest number of sets in a partition of $V$ into $k$%
-tuple dominating sets. Recall that when $k=1$, a $k$-tuple domatic number
is the well-studied domatic number.

In this work, we derive basic properties and bounds for the $k$-tuple
domatic number.
\end{abstract}



\section{ Introduction}

The notation we use is as follows. Let $G$ be a simple graph with
\emph{vertex set} $V=V(G)$ and \emph{edge set} $E=E(G)$. The
\emph{order} $\mid V\mid $ of $G$ is denoted by $n=n(G)$. For every
vertex $v\in V$, the \emph{open neighborhood} $N_{G}(v)$ is the set
$\{u\in V\mid uv\in E\}$ and the \emph{closed neighborhood}
of $v$ is the set $N_{G}[v]=N_{G}(v)\cup \{v\}$. The \emph{degree} of a vertex $%
v\in V$ is $deg(v)=\mid N(v)\mid $. The \emph{minimum} and
\emph{maximum degree} of a graph $G$ are denoted by $\delta =\delta
(G)$ and $\Delta =\Delta (G)$, respectively. If every vertex of $G$
has degree $k$, then $G$ is said to be $k$-\emph{regular}. The
\emph{complement} of a graph $G$ is denoted by $\overline{G}$ which
is a graph with $V(\overline{G})=V(G)$ and for every two vertices
$v$ and $w$, $vw\in E(\overline{G})$ if and only if $vw\notin E(G)$.
The subgraph \emph{induced} by $S$ in a graph $G$ is denoted by $G[S]$. We write $%
K_{n}$ for the \emph{complete graph} of order $n$ and $K_{n,m}$ for
the \emph{complete bipartite graph}.

For every positive integer $k$, the $k$-\emph{join} $G\circ _{k}H$
of a graph $G$ to a graph $H$, of order at least $k$, is the graph
obtained from the disjoint union of $G$ and $H$ by joining each
vertex of $G$ to at least~$k$ vertices of $H$.

A \emph{dominating set} of a graph $G$ is a subset $S$ of the vertex
set $V(G)$ such that every vertex of $G$ is either in $S$ or has a
neighbor in $S$. The minimum cardinality of a dominating set of $G$
is the \emph{domination number} $\gamma (G)$ of $G$. It is well
known that the complement of a dominating set of minimum cardinality
of a graph $G$ without isolated vertices is also a dominating set.
Hence one can partition the vertex set of $G$ into at least two
disjoint dominating sets. The maximum number of dominating sets into
which the vertex set of a graph $G$ can be partitioned is called the $domatic$ $%
number$ of $G$, and denoted by $d(G)$. This graph invariant was
introduced by Cockayne and Hedetniemi \cite{CH}. They also showed
that
\begin{equation}
\gamma (G)\cdot d(G)\leq n.  \label{eqq}
\end{equation}%
To simplify matters of notation, a domatic partition of a graph $G$ into $%
\ell $ dominating sets is given by a colouring $f:V(G)\rightarrow
\{1,2,...,\ell \}$ of the vertex set $V(G)$ with $\ell $ colors. The
dominating sets are recovered from $f$ by taking the inverse, i.e. $%
D_{i}=f^{-l}(i)$, $i=1,...,\ell $. Clearly, a coloring $f$ defines a domatic partition of $%
G $ if and only if for every vertex $x\in V(G)$,
$f(N(x))=\{1,2,...,\ell \}$. Thus, any graph $G$ satisfies $d(G)\leq
\delta (G)+1$. The word domatic, an amalgamation of the words
domination and chromatic, refers to an analogy between the chromatic
number (partitioning of the vertex set into independent sets) and
the domatic number (partitioning into dominating sets). For a survey
of results on the domatic number of graphs we refer the reader to
\cite{Zel1}. It was first observed by Cockayne and Hedetniemi
\cite{CH} that for every graph without isolated vertices
$2\leq d(G)\leq \delta (G)+1$. The upper bound $\delta (G)+1$\ is attained for interval graphs \cite%
{LHC}, for example.

Intuitively, it seems reasonable to expect that a graph with large minimum
degree will have a large domatic number. Zelinka \cite{Zel2} showed that
this is not necessarily the case. He gave examples for graphs of arbitrarily
large minimum degree with domatic number $2$. For more details about domatic
number see the references \cite{Chen}, \cite{DN}, \cite{RV} and \cite{RRS}.

The \emph{total domatic number} $d_t(G)$ is similarly defined based
on the concept of the \emph{total domination number} $\gamma _t(G)$.
Sheikholeslami and Volkmann, in a similar manner, generalized in
\cite{SV} the concept of total domatic number to the $k$-tuple total
domatic number $d_{\times k,t}(G)$ based on the concept of $k$-tuple
total domination number $\gamma_{\times k,t}(G)$, which is defined
by Henning and Kazemi in \cite {HK}. We recall that for every
positive integer $k$, a $k$-\emph{tuple total dominating set},
abbreviated kTDS, of a graph $G$ is a subset $S$ of the vertex set
$V(G)$ such that every vertex of $G$ is adjacent to at least $k$
vertices of $S$. And the minimum cardinality of a kTDS of $G$ is the
$k$-\emph{tuple total domination number} $\gamma _{\times k,t}(G)$
of $G$.

Here, we extend the concept of domatic number to $k$-tuple domatic
number based on the concept of $k$-tuple domination number, which is
defined by Harary and Haynes in \cite {HH00}. For every positive
integer $k$, a $k$-\emph{tuple dominating set}, abbreviated kDS, of
a graph $G$ is a subset $S$ of the vertex set $V(G)$ such that every
vertex of $G$ is either in $S$ and is adjacent to at least $k-1$
vertices of $S$ or is not in $S$ and is adjacent to at least $k$
vertices of $S$. The minimum cardinality of a kDS of $G$ is the
$k$-\emph{tuple domination number} $\gamma _{\times k}(G)$ of $G$.
For a graph to have a $k$-tuple dominating set, its minimum degree
is at least $k-1$. The $k$-\emph{tuple domatic number} $d_{\times
k}(G)$ of $G$ is the largest number of sets in a partition of $V(G)$
into $k$-tuple dominating sets. If $d=d_{\times k}(G)$ and
$V(G)=V_{1}\cup V_{2}\cup ...\cup V_{d}$ is a partition of $V(G)$
into $k$-tuple dominating sets $V_{1}$, $V_{2}$, ... and $V_{d}$, we
say that $\{V_{1},V_{2},...,V_{d}\}$\ is a
$k$-\emph{tuple domatic partition}, abbreviated kDP, of $G$. The $k$%
-tuple domatic number is well-defined and
\begin{equation}
d_{\times k}(G)\geq 1,
\end{equation}%
for all graphs $G$ with $\delta (G)\geq k-1$, since the set consisting of $%
V(G)$ forms a $k$-tuple domatic partition of $G$.

To simplify matters of notation, a $k$-tuple domatic partition of a graph $G$
into $\ell $ $k$-tuple dominating sets is given by a coloring $%
f:V(G)\rightarrow \{1,2,...,\ell \}$ of the vertex set $V(G)$ with $\ell $
colors. The $k$-tuple dominating sets are recovered from $f$ by taking the
inverse, i.e. $D_{i}=f^{-l}(i)$, $i=1,...,\ell $. Clearly, a coloring $f$ defines a $k$%
-tuple domatic partition of $G$ if and only if for every vertex $x\in V(G)$%
, $f(N(x))=\{f(y)\mid y\in N(x)\}$ contains the mulitiset $%
\{t_{1}.1,t_{2}.2,...,t_{\ell }.\ell \}$ such that for every $i$,
$t_{i}\in \{k-1,k\}$ and for an index $i$, if $t_{i}=k-1$, then
$f(x)=i$. Clearly, each graph $G$ satisfies%
\begin{equation}
d_{\times k}(G)\leq \frac {\delta (G)+1}{k}.
\end{equation}
Graphs for which $d_{\times k}(G)$ achieves this upper bound
$\frac{\delta (G)+1}{k}$ we call $k$-\emph{tuple domatically full}.

In this work, we derive basic properties and bounds for the
$k$-tuple domatic number.

The following observations are useful.

\begin{observation}
\label{Ob.1} Let $K_{n}$ be the complete graph of order $n\geq 1$.
Then
\begin{equation*}
d_{\times k}(K_{n})=\lfloor \frac{n}{k}\rfloor \mathrm{.}
\end{equation*}
\end{observation}

\begin{observation}
\label{Obs.bipartite} Let $G$ be a bipartite graph with $\delta
(G)\geq k-1\geq 1$. If $X$ and $Y$ are the bipartite sets of $G$,
then $\gamma _{\times k}(G)\geq 2k-2$ with equality if and only if
$G=K_{k-1,k-1}$.
\end{observation}

\begin{proof}
Let $D$ be a $\gamma_{\times k}(G)$-set, and let $w\in X$ and $z\in
Y$ be two arbitrary vertices. The definition implies that $\mid
D\cap N(w)\mid\geq k-1$ and $\mid D\cap N(z)\mid\geq k-1$. Since
$N(w) \cap N(z)=\emptyset$, we deduce that $\mid D\mid \geq 2k-2$
and thus $\gamma _{\times k}(G)\geq 2k-2$. Obviously, we can see
that $\gamma _{\times k}(G)=2k-2$ if and only if $G=K_{k-1,k-1}$.
\end{proof}

\section{properties of the $k$-tuple domatic number}

Here, we present basic properties of $d_{\times k}(G)$ and bounds on the $k$%
-tuple domatic number of a graph. We start our work with a theorem
that characterizes graphs $G$ with $\gamma _{\times k}(G)=m$, for
some $m\geq k-1$.

\begin{theorem}
\label{kDm} Let $G$ be a graph with $\delta (G)\geq k-1$. Then for
any integer $m\geq k-1$, $\gamma _{\times k}(G)=m$ if and only if
$G=K_{m}^{\prime }$ or $G=F\circ _{k}K_{m}^{\prime }$, for some
graph $F$ and some spanning subgraph $K_{m}^{\prime }$\ of $K_{m}$\
with $\delta (K_{m}^{\prime })\geq k-1$  such that $m$ is minimum in
the set
\begin{equation}
\{t\mid G=F^{\prime } \circ _{k}K_{t}^{\prime },\mbox{ for some
graph }F^{\prime } \mbox{ and some spanning subgraph }K_{t}^{\prime
}\mbox{ of }K_{t} \mbox{ with } \delta (K_{t}^{\prime })\geq k-1\}.
\end{equation}%
\end{theorem}

\begin{proof}
Let $S$ be a $\gamma _{\times k}(G)$-set and $\gamma _{\times
k}(G)=m$, for some $m\geq k-1$. Then, $\mid S\mid =m$ and every
vertex in $V-S$ has at least $k$ neighbors in $S$ and otherwise
$k-1$ neighbors. Then $G[S]=K_{m}^{\prime }$, for some spanning
subgraph $K_{m}^{\prime }$\ of $K_{m}$ with $\delta (K_{m}^{\prime
})\geq k-1$. If $\mid V\mid =m$, then $G=K_{m}^{\prime }$. If $\mid
V\mid
>m$, then let $F$ be the induced subgraph $G[V-S]$. Then $G=F\circ _{k}K_{m}^{\prime }$.
Also by the definition of $k$-tuple domination number, $m$ is
minimum in the set given in (4).

Conversely, let $G=K_{m}^{\prime }$ or $G=F\circ _{k}K_{m}^{\prime
},$ for some graph $F$ and some spanning subgraph $K_{m}^{\prime }$\
of $K_{m}$\ with $\delta (K_{m}^{\prime })\geq k-1$ such that $m$ is
minimum in the set given in (4). Then, since $V(K_{m}^{\prime })$ is
a kDS of $G$ with cardinal $m$, $\gamma _{\times k}(G)\leq m$. If
$\gamma _{\times k}(G)=m^{\prime }<m$,
then the previous paragraph concludes that for some graph $%
F^{\prime }$ and some spanning subgraph $K_{m^{\prime }}^{\prime }$\
of $K_{m^{\prime }}$ with $\delta (K_{m^{\prime }}^{\prime })\geq
k-1$, $G=F^{\prime }\circ _{k}K_{m^{\prime }}^{\prime }$, that is
contradiction with the minimality of $m$. Therefore $\gamma_{\times
k}(G)=m$.
\end{proof}

\begin{corollary}
\label{kDm=k-1} Let $G$ be a graph with $\delta (G)\geq k-1$. Then
$\gamma _{\times k}(G)=k-1$ if and only if $G=K_{k-1}$ or $G=F\circ
_{k}K_{k-1}$, for some graph $F$.
\end{corollary}

\begin{theorem}
\label{Ob.G.D} If $G$ is a graph of order $n$ and $\delta (G)\geq k-1$, then
\begin{equation*}
\gamma _{\times k}(G)\cdot d_{\times k}(G)\leq n.
\end{equation*}
Moreover, if $\gamma _{\times k}(G)\cdot d_{\times k}(G)=n$, then for each
kDP $\{V_{1},V_{2},...,V_{d}\}$ of $G$ with $d=d_{\times k}(G)$, each set $%
V_{i}$ is a $\gamma _{\times k}(G)$-set.
\end{theorem}

\begin{proof}
Let $\{V_{1},V_{2},...,V_{d}\}$ be a kDP of $G$ such that $d=d_{\times k}(G)$%
. Then
\begin{equation*}
\begin{array}{lll}
d\cdot \gamma _{\times k}(G) & = & \sum\limits_{i=1}^{d}\gamma _{\times k}(G)
\\
& \leq & \sum\limits_{i=1}^{d}\mid V_{i}\mid \\
& = & n.%
\end{array}%
\end{equation*}
If $\gamma _{\times k}(G)\cdot d_{\times k}(G)=n$, then the inequality
occurring in the proof becomes equality. Hence for the kDP $%
\{V_{1},V_{2},...,V_{d}\}$ of $G$ and for each $i$, $\mid V_{i}\mid =\gamma
_{\times k}(G)$. Thus each set $V_{i}$ is a\ $\gamma _{\times k}(G)$-set.
\end{proof}

The case $k=1$ in Theorem \ref{Ob.G.D} leads to the well-known inequality $%
(1)$, given by Cockayne and Hedetniemi \cite{CH} in 1977.

An immediate consequence of Corollary \ref{kDm=k-1} and Theorem
\ref{Ob.G.D} now follows.

\begin{corollary}
\label{Ob.Cor.D} If $G$ is a graph of order $n$ with $\delta (G)\geq
k-1\geq 1$, then
\begin{equation*}
d_{\times k}(G)\leq \frac{n}{k-1},
\end{equation*}
with equality if and only if $G=K_{k-1}$ or $G=F\circ _{k}K_{k-1}$,
for some graph $F$.
\end{corollary}

For bipartite graphs, we can improve the bound given in Corollary
\ref{Ob.Cor.D}, by Observation \ref{Obs.bipartite}.

\begin{corollary}
\label{d<=n/2k} Let $G$ be a bipartite graph of order $n$ with
vertex partition $V(G)=X\cup Y$ and $\delta (G)\geq k-1\geq 1$. Then
\begin{equation*}
d_{\times k}(G)\leq \frac{n}{2k-2},
\end{equation*}
with equality if and only if $G=K_{k-1,k-1}$.
\end{corollary}

\begin{theorem}
\label{Ob.G+D} If $G$ is a graph of order $n$ and $\delta (G)\geq k-1\geq 2$%
, then
\begin{equation*}
\gamma _{\times k}(G)+d_{\times k}(G)\leq n+1.
\end{equation*}
\end{theorem}

\begin{proof}
Applying Theorem \ref{Ob.G.D}, we obtain
\begin{equation*}
\gamma _{\times k}(G)+d_{\times k}(G)\leq \frac{n}{d_{\times k}(G)}%
+d_{\times k}(G).
\end{equation*}%
Since $d_{\times k}(G)\geq 1$, by inequality (2), and $k\geq 3$,
Corollary \ref{Ob.Cor.D} implies that $d_{\times k}(G)\leq
\frac{n}{2}$. Using these inequalities, and the fact that the
function $g(x)=x+\frac{n}{x}$ is decreasing for $1\leq x\leq n^{1/2}
$ and increasing for $n^{1/2}\leq x\leq \frac{n}{2}$, we obtain
\begin{equation*}
\gamma _{\times k}(G)+d_{\times k}(G)\leq \max \{n+1,\frac{n}{2}+2\}=n+1,
\end{equation*}%
and this is the desired bound.
\end{proof}

If $G=\ell K_{k}$ for integers $\ell \geq 1$ and $k\geq 3$, then
$\gamma _{\times k}(G)=n(G)=\ell k$ and $d_{\times k}(G)=1$.
Therefore $\gamma _{\times k}(G)+d_{\times k}(G)=n+1$, and so the
upper bound $n+1$ in Theorem \ref{Ob.G+D} is sharp.

By closer look at the proof of Theorem \ref{Ob.G+D} we have:

\begin{theorem}
\label{Ob.G+D.2}Let $G$ be a graph of order $n$ with $\delta (G)\geq
k-1\geq 2$. If $d_{\times k}(G)\geq 2$, then
\begin{equation*}
\gamma _{\times k}(G)+d_{\times k}(G)\leq \frac{n}{2}+2.
\end{equation*}
\end{theorem}

If $G=K_{2k}$, then $\gamma _{\times k}(G)=k$ and $d_{\times k}(G)=2$.
Therefore $\gamma _{\times k}(G)+d_{\times k}(G)=n/2+2$, and so the upper
bound $n/2+2$\ in Theorem \ref{Ob.G+D.2} is sharp.

\begin{theorem}
\label{Ob.D.upper} If $G$ is a graph with $\delta (G)\geq k-1$, then
\begin{equation*}
d_{\times k}(G)\leq \frac{\delta (G)+1}{k}.
\end{equation*}%
This bound is sharp and moreover,\ if $d_{\times k}(G)=(\delta
(G)+1)/k$, then for each kDP $\{V_{1},V_{2},...,V_{d}\}$ of $G$ with
$d=d_{\times k}(G)$ and for all vertices $v$ of degree $\delta (G)$,
$\mid V_{i}\cap N_{G}[v]\mid =k$ for each $1\leq i\leq d$.
\end{theorem}

\begin{proof}
Let $\{V_{1},V_{2},...,V_{d}\}$ be a kDP of $G$ such that
$d=d_{\times k}(G)$, and let $v$ be a vertex of degree $\delta (G)$.
Since $\mid V_{i}\cap N_{G}[v]\mid \geq k$ for each $1\leq i\leq d$,
then
\begin{equation*}
\begin{array}{lll}
k\cdot d_{\times k}(G) & \leq & \sum\limits_{i=1}^{d}\mid V_{i}\cap N_{G}[v]\mid \\
& = & \mid N_{G}[v]\mid \\
& = & \delta (G)+1,%
\end{array}%
\end{equation*}
as desired. This bound is sharp for the complete graphs which their
orders are multiple of $k$. Since $d_{\times k}(G)=(\delta (G)+1)/k$
follows that the inequality occurring in the above becomes equality,
which leads to the property given in the statement.
\end{proof}

\begin{corollary}
\label{Ob.cor.2} Let $k\geq 1$ be an integer, and let $G$ be a graph. If $%
k-1\leq \delta (G)\leq 2k-2$, then $d_{\times k}(G)=1$.
\end{corollary}

As a further application of Theorem \ref{Ob.D.upper}, we will prove the
following result.

\begin{theorem}
\label{Ob.D+D}For every graph $G$ of order $n$ in which $min\{\delta
(G),\delta (\overline{G})\}\geq k-1$,
\begin{equation*}
d_{\times k}(G)+d_{\times k}(\overline{G})\leq \frac{n+1}{k},
\end{equation*}%
and this bound is sharp.
\end{theorem}

\begin{proof}
Theorem \ref{Ob.D.upper} follows that
\begin{equation*}
\begin{array}{lll}
d_{\times k}(G)+d_{\times k}(\overline{G}) & \leq  & \frac{\delta (G)+\delta
(\overline{G})+2}{k} \\
& = & \frac{(\delta (G)+1)+(n-\Delta (G))}{k} \\
& \leq  & \frac{n+1}{k},%
\end{array}%
\end{equation*}%
as desired.

If $G$ is the complete bipartite graph $K_{k,k}$, where $k\geq
2 $, then $d_{\times k}(G)+d_{\times k}(\overline{G})=1+1=\lfloor \frac{2k+1%
}{k}\rfloor $, and so the upper bound $\frac{n+1}{k}$\ is sharp.
\end{proof}

Now we derive some structural properties on graphs with equality in the
inequality of Theorem \ref{Ob.D+D}.

\begin{theorem}
\label{Ob.D+D=}Let $G$ be a graph of order $n$ with $min\{\delta (G),\delta (%
\overline{G})\}\geq k-1$ which
\begin{equation*}
d_{\times k}(G)+d_{\times k}(\overline{G})=\frac{n+1}{k},
\end{equation*}%
and $d_{\times k}(G)\geq d_{\times k}(\overline{G})$. Then G is regular and
\begin{equation*}
\frac{n}{r+1}+\frac{1}{k}\leq d_{\times k}(G)\leq \frac{n}{r}
\end{equation*}%
for an integer $k-1\leq r\leq 2k-1$.
\end{theorem}

\begin{proof}
According to Theorem \ref{Ob.D.upper}, we have
\begin{equation*}
d_{\times k}(G)+d_{\times k}(\overline{G})\leq \frac{\delta (G)+\delta (%
\overline{G})+2}{k}.
\end{equation*}%
If $G$ is not regular, then $\delta (G)+\delta (\overline{G})\leq
n-2$, and we obtain the upper bound $d_{\times k}(G)+d_{\times
k}(\overline{G})\leq \frac{n}{k}<\frac{n+1}{k}$, a contradiction.
Thus $G$ is regular.

The hypothesis $d_{\times k}(G)\geq d_{\times k}(\overline{G})$ and the
hypothesis $d_{\times k}(G)+d_{\times k}(\overline{G})=\frac{n+1}{k}$ lead
to
\begin{equation*}
d_{\times k}(G)\geq \frac{n+1}{2k}.
\end{equation*}%
Let $\{V_{1},V_{2},...,V_{d}\}$ be a kDP of $G$ such that
$d=d_{\times k}(G)$ and $r=\mid V_{1}\mid \leq \mid V_{2}\mid \leq
...\leq \mid V_{d}\mid $. Clearly, $r\geq k-1$ and
\begin{equation*}
r\cdot d_{\times k}(G)\leq n.
\end{equation*}%
If $r\geq 2k$, then
\begin{equation*}
\begin{array}{lll}
n & \geq  & r\cdot d \\
& \geq  & 2k\cdot \frac{n+1}{2k} \\
& > & n.%
\end{array}%
\end{equation*}%
Therefore we have shown that $k-1 \leq r \leq 2k-1$. Since $V_{1}$ is a $k$%
-tuple dominating set and $G$ is regular, we deduce that
\begin{equation*}
\begin{array}{lll}
r\cdot \Delta (G) & = & \sum\nolimits_{v\in V_{1}}\deg (v) \\
& \geq  & k(n-r)+(k-1)r \\
& = & kn-r%
\end{array}%
\end{equation*}%
and thus $\Delta (G)=\delta (G)\geq \frac{kn}{r+1}$ and so
\begin{equation*}
\begin{array}{lll}
\delta (\overline{G})+1 & = & n-\delta (G) \\
& \leq  & n-\frac{kn}{r+1} \\
& = & \frac{n(r+1)-kn}{r+1}.%
\end{array}%
\end{equation*}%
Applying Theorem \ref{Ob.D.upper}, we thus obtain
\begin{equation*}
\begin{array}{lll}
d_{\times k}(\overline{G}) & \leq  & \frac{\delta (\overline{G})+1}{k} \\
& \leq  & \frac{n(r+1)-kn}{k(r+1)}.%
\end{array}%
\end{equation*}%
Now $d_{\times k}(G)+d_{\times k}(\overline{G})=\frac{n+1}{k}$\ leads to
\begin{equation*}
\begin{array}{lll}
d_{\times k}(G) & = & \frac{n+1}{k}-d_{\times k}(\overline{G}) \\
& \geq  & \frac{n}{r+1}+\frac{1}{k}.%
\end{array}%
\end{equation*}
\end{proof}

\begin{corollary}
\label{Ob.Cor.D=}Let $G$ be a graph of order $n$ with $min\{\delta (G),\delta (%
\overline{G})\}\geq k-1\geq 1$ which
\begin{equation*}
d_{\times k}(G)+d_{\times k}(\overline{G})=\frac{n+1}{k},
\end{equation*}%
and $d_{\times k}(G)\geq d_{\times k}(\overline{G})$. Then
\begin{equation*}
\frac{n}{2k}+\frac{1}{k}\leq d_{\times k}(G)\leq \frac{n}{k-1}.
\end{equation*}
\end{corollary}

We now present a sharp lower bound on the $k$-tuple domatic number,
which generalizes the bound due to Zelinka \cite{Zel2} in 1983.

\begin{theorem}
\label{Ob.D.Lower}For every graph $G$ of order $n$ with $\delta
(G)\geq k-1$,
\begin{equation*}
d_{\times k}(G)\geq \lfloor \frac{n}{k(n-\delta (G))}\rfloor ,
\end{equation*}%
and this bound is sharp.
\end{theorem}

\begin{proof}
If $k(n-\delta (G))>n$, then there is nothing to prove. Thus we assume in
the following that $n\geq k(n-\delta (G))$. Now let $S\subseteq V(G)$ be any
subset with $\mid S\mid \geq k(n-\delta (G))$. It follows that
\begin{equation*}
\mid S\mid \geq k(n-\delta (G))\geq n-\delta (G)+k-1
\end{equation*}
and therefore $\mid V(G)-S\mid \leq \delta (G)-k+1$. This inequality implies
that
\begin{equation*}
\mid N_{G}(u)\cap S\mid \geq \delta (G)-(\delta (G)-k)=k
\end{equation*}
for $u\in V(G)-S$ and
\begin{equation*}
\mid N_{G}(u)\cap S\mid \geq \delta (G)-(\delta (G)-k+1)=k-1
\end{equation*}
for $u\in S$. Hence $S$ is a $k$-tuple dominating set of $G$. Let $n=\ell
k(n-\delta (G))+r$ with integers $\ell \geq 1$ and $0\leq r\leq k(n-\delta
(G))-1$, then one can take any $\ell $ disjoint subsets, $\ell -1$ of
cardinality $k(n-\delta (G))$ and one of cardinality $k(n-\delta (G))+r$,
and all these subsets are $k$-tuple dominating sets of $G$. This yields a $k$%
-tuple domatic partition of cardinality $\ell =\lfloor \frac{n}{k(n-\delta
(G))}\rfloor $, and thus our Theorem is proved.

We also note that this lower bound is sharp for the complete graph $K_{\ell
k}$.
\end{proof}

\begin{corollary}
\label{Ob.Cor.D.Lower} \cite{Zel2} For every graph $G$ of order $n$,
\begin{equation*}
d(G)\geq \lfloor \frac{n}{n-\delta (G)}\rfloor .
\end{equation*}
\end{corollary}

Finally, we compare the $k$-tuple domatic number of a graph with its
$k$-tuple total domatic number.

\begin{theorem}
\label{d_xk,d_xk,t} Let $G$ be a graph with $\delta (G)\geq k\geq
1$. Then
\begin{equation*}
d_{\times k,t}(G)\leq d_{\times k}(G)\leq 2d_{\times k,t}(G),
\end{equation*}
and this bounds are sharp.
\end{theorem}

\begin{proof}
Since every $k$-tuple total dominating set of $G$ is a $k$-tuple
dominating set and the union of at least two disjoint $k$-tuple
dominating sets is a $k$-tuple total dominating set, then $d_{\times
k,t}(G)\leq d_{\times k}(G)\leq 2d_{\times k,t}(G)$.

The lower bound is sharp for the complete bipartite graph
$K_{mk,mk}$, where $k\geq 2$ and $m\geq 1$. Because $d_{\times
k,t}(G)=d_{\times k}(G)=m$. Also for the cycle $C_4$, we have
$d(C_4)=d_t(C_4)=2$.

The upper bound is sharp for the graphs $G$ which is obtained as
follow: let $H_1$, $H_2$, $H_3$ and $H_4$ be four disjoint copies of
the complete graph $K_k$, where $k\geq 1$. Let $G$ be the union of
the four graphs $H_1$, $H_2$, $H_3$ and $H_4$ such that for each
$1\leq i\leq 3$ every vertex of $H_i$ is adjacent to all vertices of
$H_{i+1}$. Obviously $V(H_2)\cup V(H_3)$ is the unique $\gamma
_{\times k,t}(G)$-set, and so $d_{\times k,t}(G)=1$. This follows
that $d_{\times k}(G)\leq 2d_{\times k,t}(G)=2$. Since the sets
$V(H_2)\cup V(H_3)$ and $V(H_1)\cup V(H_4)$ are two disjoint $\gamma
_{\times k}(G)$-sets, then $d_{\times k}(G)=2=2d_{\times k,t}(G)$.
\end{proof}

\begin{corollary}
\label{d,d_t} \cite{Zel3} Let $G$ be a graph with no isolated
vertices. Then
\begin{equation*}
d_{t}(G)\leq d(G)\leq 2d_t(G).
\end{equation*}
\end{corollary}

\begin{theorem}
\label{d_xk=d_xk,t} Let $k\geq 1$ be integer. If one of the numbers
$d_{\times k}(G)$ and $d_{\times k,t}(G)$ for a graph $G$ is
infinite, then
\begin{equation*}
d_{\times k}(G)=d_{\times k,t}(G).
\end{equation*}
\end{theorem}

\begin{proof}
Let $d_{\times k}(G)=\alpha$, where $\alpha$ is an infinite cardinal
number. Then there exists a $k$-tuple domatic partition $\Re$ having
$\alpha$ classes. The family $\Re$ can be partitioned into two
subfamilies $\Re _1$ and $\Re _2$ which both have the cardinality
$\alpha$. There exists a bijection $f:\Re _1 \rightarrow \Re _2$.
Let $\Re _0=\{D \cup f(D) \mid D\in \Re _1\}$. This is evidently a
$k$-tuple total domatic partition of $G$ having $\alpha$ classes and
thus $d_{\times k,t}(G)\geq \alpha =d_{\times k}(G)$. Since
$d_{\times k,t}(G)\leq d_{\times k}(G)$, we have $d_{\times
k,t}(G)=d_{\times k}(G)=\alpha$. If $d_{\times k,t}(G)$ is infinite,
then so is $d_{\times k}(G)$ and also $d_{\times k,t}(G)=d_{\times
k}(G)$.
\end{proof}

\begin{corollary}
\label{d=d_t} \cite{Zel4} If one of the numbers $d(G)$ and $d_t(G)$
for a graph $G$ is infinite, then
\begin{equation*}
d(G)=d_t(G).
\end{equation*}
\end{corollary}


\end{document}